\documentclass[oneside,a4paper,reqno,psamsfonts]{amsart}
\usepackage[latin1]{inputenc} 
\usepackage{amscd}
\usepackage{amsopn}
\usepackage{amstext}
\usepackage{amsxtra}
\usepackage{amssymb}
\usepackage{stmaryrd}
\usepackage{mathrsfs}
\usepackage{upref}
\usepackage{textcase} 
\usepackage{clrscode}
\usepackage{bm} 
\usepackage{graphicx}
\usepackage{setspace}
\newcounter{algorithm}

\theoremstyle{plain}
\newtheorem{thm}{Theorem}
\newtheorem{lem}[thm]{Lemma}
\newtheorem{cl}[thm]{Corollary}

\theoremstyle{definition}

\theoremstyle{remark}

\providecommand{\abs}[1]{\left\lvert #1 \right\rvert}

\providecommand{\set}[1]{\left\lbrace #1 \right\rbrace}
\providecommand{\gen}[1]{\left\langle #1 \right\rangle}

\newcommand{\field}[1]{\mathbb{#1}}

\newcommand{\N}{\field{N}}

\newcommand{\F}{\field{F}}

\newcommand{\MAGMA}{\textsc{Magma}}
\def\cyc#1{\mathord{\mathrm{C}_{#1}}}
\def\dih#1{\mathord{\mathrm{D}_{#1}}}
\def\cen{\mathord{\mathrm{C}}}
\def\nrm{\mathord{\mathrm{N}}}
\let\le=\leqslant
\let\ge=\geqslant

\newcommand{\OV}{\mathcal{O}}

\DeclareMathOperator{\GL}{GL}

\DeclareMathOperator{\I}{Id}

\DeclareMathOperator{\Sz}{Sz}

\DeclareMathOperator{\diag}{diag}
\DeclareMathOperator{\antidiag}{antidiag}
\DeclareMathOperator{\Tr}{Tr}
\DeclareMathOperator{\SLP}{\mathrm{SLP}}

\DeclareMathOperator{\Cent}{C}

\DeclareMathOperator{\O2}{O_2}

\DeclareMathOperator{\Mat}{Mat}

\newcommand{\OR}[1]{\operatorname{O} \bigl( #1 \bigr)}

\title{A new method for recognising Suzuki groups}
\author{John N. Bray}
\email{j.n.bray@qmul.ac.uk}
\address{School of Mathematical Sciences \\ Queen Mary, University of London \\ Mile End Road \\ London E1 4NS \\ United Kingdom}
\author{Henrik B\"a\"arnhielm}
\email{henrik@math.auckland.ac.nz}
\address{Department of Mathematics \\ University of Auckland \\ Private Bag 92019 \\ Auckland \\ New Zealand}

\begin{document}

\begin{abstract}
We present a new algorithm for constructive recognition of the Suzuki groups in their natural representations.
The algorithm runs in Las Vegas polynomial time given a discrete logarithm oracle. An implementation is available in the $\MAGMA$ computer algebra system.
\end{abstract}

\maketitle

\section{Introduction}
\label{section:intro}

In \cite{baarnhielm_phd} and \cite{baarnhielm05}, algorithms for
constructive recognition of the Suzuki groups in the natural
representation are presented. They depend on a technical conjecture,
which is still open, although supported by substantial experimental evidence.

Here we present a new algorithm for this problem, which does not
depend on any such conjectures, and which is also more efficient.

We shall use the notation of \cite{baarnhielm05}, but for completeness we state the important points here.
The ground finite field is $\F_q$ where $q = 2^{2m + 1}$ for some $m > 0$,
and we define $t = 2^{m + 1}$ so that $x^{t^2} = x^2$ for every $x \in \F_q$.
For $a, b \in \F_q$ and $\lambda \in \F_q^{\times}$, define the following matrices.
\begin{equation}
U(a, b) = \begin{bmatrix}
1 & 0 & 0 & 0\ \\
a & 1 & 0 & 0\ \\
a^{t+1} + b & a^t & 1 & 0\ \\
a^{t + 2} + ab + b^t & b & a & 1\ 
\end{bmatrix}, 
\end{equation}
\begin{equation}
M^{\prime}(\lambda) = \begin{bmatrix}
\lambda^{t+1} & 0 & 0 & 0 \\
0 & \lambda & 0 & 0 \\
0 & 0 & \lambda^{-1} & 0 \\
0 & 0 & 0 & \lambda^{-t-1}
\end{bmatrix}, 
\end{equation}
\begin{equation}
T = \begin{bmatrix}
\ 0 & 0 & 0 & 1\ \\
\ 0 & 0 & 1 & 0\ \\
\ 0 & 1 & 0 & 0\ \\
\ 1 & 0 & 0 & 0\  
\end{bmatrix}.
\end{equation}
If $\omega \in \F_q$ is a primitive element, then $\Sz(q) = \gen{U(1, 0), M^{\prime}(\omega), T}$.
This is our standard copy of $\Sz(q)$, denoted $\Sigma$.
This group acts on the Suzuki ovoid, which is
\begin{equation} \label{ovoid_def}
\OV = \set{(1 : 0 : 0 : 0)} \cup \set{\,(a^{t + 2} + a b + b^t : b : a : 1) \mid a,b \in \F_q\,}.
\end{equation}
Let $\mathcal{F} = \set{\,U(a, b) \mid a, b \in \F_q\,}$ and
$\mathcal{H} = \set{\,M^{\prime}(\lambda) \mid \lambda \in \F_q^{\times}\,}$.
Then $\mathcal{F}\mathcal{H}=\mathcal{H}\mathcal{F}$ is the stabiliser of $(1 : 0 : 0 : 0) \in \OV$,
a maximal subgroup of $\Sz(q)$ and  $\mathcal{F}\mathcal{H} = \gen{U(1, 0), M^{\prime}(\omega)}
\cong \F_q . \F_q . \F_q^{\times}$. The group $\Sz(q)$ is partitioned into two sets as
\begin{equation} \label{sz_partition}
\Sz(q) = \mathcal{F}\mathcal{H} \cup \mathcal{F}\mathcal{H}T\mathcal{F}
 = \mathcal{H}\mathcal{F} \cup \mathcal{H}\mathcal{F}T\mathcal{F}.
\end{equation}

If $G$ is a conjugate of $\Sz(q)$, so that $G^c = \Sz(q)$ for some $c \in \GL(4, q)$,
we say that the ordered triple of elements $\alpha, h, \gamma \in G$ are
\emph{rewriting generators for $G$ with respect to $c$} if
\begin{itemize}
\item $\alpha^c \in \mathcal{F}$, $h^c \in \mathcal{F}\mathcal{H}$, $\gamma^c = T$,
\item $\alpha$ has order $4$ and $h$ has odd order not dividing $r-1$ for any
$r$ such that $q$ is a non-trivial power of $r$.
\end{itemize}
Note that these conditions imply that $G = \gen{\alpha, h, \gamma}$.

The main results are Theorem \ref{main_theorem} and \ref{main_theorem2}. The following result is a consequence:

\begin{thm} Given a random element oracle for subgroups of\/
  $\GL(4, q)$ and an oracle for the discrete logarithm problem in
  $\F_q$:

\begin{itemize}
\item There exists a Las Vegas algorithm that, given a conjugate
  $G = \gen{X}$ of\/ $\Sigma$, constructs $g \in \GL(4, q)$ such that
  $\gen{X}^g = \Sigma$ and also constructs rewriting generators for $G$
  with respect to $g$ as $\SLP$s in $X$.
  The algorithm has expected time complexity\/
  $\OR{\log(q)\log\log(q)}$ field operations. 
\item Let $h\in \GL(4,F)$. Then:
\begin{enumerate}
\item We can determine whether $h \in G$ in time $\OR{1}$ field operations.
\item Given some preprocessing that is independent of $h$, and given
  that $h \in G$, we can construct $h$ as an\/ $\SLP$ in $X$ in time
  $\OR{\log(q)}$ field operations. The preprocessing, which only has
  to be done once per constructive recognition of $G$, has complexity
  $\OR{(\log(q))^2}$.
\end{enumerate}
\end{itemize}
\end{thm}

\section{Overview}

The group $G = \Sz(q)$ acts doubly transitively on $\OV$. In
\cite[Corollary 4.8]{baarnhielm05}, an algorithm for constructing a
generating set of a stabiliser of a given point $P \in \OV$ is
described. The generators are constructed as $\SLP$s in a given
generating set $X$ for $\Sz(q)$.

This is then used in \cite[Theorem 5.2]{baarnhielm05} to obtain
generating sets for the stabilisers of the two points $P_{\infty} = (1 : 0 : 0 :0)$ and
$P_0 = (0 : 0 : 0 : 1)$, which consist of lower and upper
triangular matrices, respectively. These stabilisers are used in
constructive membership testing to convert any given $g \in \Sz(q)$ to
an $\SLP$ in $X$, given as \cite[Algorithm 2]{baarnhielm05}. 

In the natural representation, the constructive recognition problem
reduces to the construction of a conjugating matrix, and \cite[Theorem
5.2]{baarnhielm05} is also used.

The new algorithm presented here replaces the central task of
constructing a point stabiliser with a new method. A difference
compared to the old method is that the stabilised point of $\OV$ is
not part of the input, instead a random point stabiliser is
constructed. We therefore obtain two random point stabilisers of the
input group, but it turns out that we can construct a matrix that
conjugates the input group to $G$ and also the two point stabilisers
to $G_{P_{\infty}}$ and $G_{P_0}$. This is done using a simplified version of \cite[Lemma 7.4]{baarnhielm05}. After that we perform
constructive membership testing with a method similar to \cite[Theorem
5.2]{baarnhielm05}, but here we present a deterministic algorithm
instead of a probabilistic one.

\section{Constructing a point stabiliser}

The algorithm consists of two parts. The second part is an adaptation of
\cite[Theorem 17]{sz_std_gens} for the natural representation.
This is given as Theorem \ref{thm_bray_stab_trick}.

The first part is the algorithm that constructs an element of order
$4$, to be used as input to the algorithm in Theorem
\ref{thm_bray_stab_trick}. This is presented as Algorithm
\ref{alg:new_sz_trick}. We use $\Mat_n(F)$ to denote the algebra of
$n$ by $n$ matrices over the field $F$. To calculate the pseudo-order
of a matrix, we use \cite{crlg95} (see also \cite[Section
2.2]{classical_recognise}). Note that we do not need to compute
precise orders, hence large integer factorisation can be avoided. On
the other hand, it turns out that one can use integer factorisation to
avoid some potential discrete logarithm calculations.

We make use of a discrete logarithm oracle \textsc{DiscreteLog}$(\lambda, a)$,
which returns $k \ge 0$ such that $\lambda^k = a$, or $-1$
if no such number exists, that is if $a \notin \gen{\lambda}
\le \F_q^{\times}$. We also make use of a random element oracle
\textsc{Random}$(G)$ that returns independent random elements of a group
$G = \gen{X}$ as $\SLP$s in $X$. This is polynomial time by \cite{babai91};
in practice we use the product replacement algorithm \cite{lg95}, which is also polynomial time by \cite{Pak00}.

We construct an element of order $4$ by constructing elements of trace $0$.
In $G$, the elements of trace $0$ have orders $1$, $2$ or $4$.

\begin{codebox}
\refstepcounter{algorithm}
\label{alg:new_sz_trick}
\Procname{\kw{Algorithm} \ref{alg:new_sz_trick}: $\proc{Order4Elt}(G)$}
\zi \kw{Input}: Group $G = \gen{X}$ where $X \subseteq \GL(4, q)$ such that $G \cong \Sz(q)$.
\zi \kw{Output}: Element $f \in G$ as $\SLP$ in $X$, such that $\abs{f} = 4$.
\li $g := \proc{Random}(G)$ \label{line:producing_g}
\li \If $\abs{g} \mid q-1$ and $g \neq 1$ \Then \label{line:finding_g}\label{line:order_g} 
\li     $u, c, \lambda := \proc{Diagonalise}(g)$ \label{line:diagonalise}
\li     \Comment{Now $u = g^c$ where $u = \diag(\lambda^{t + 1}, \lambda, \lambda^{-1}, \lambda^{-t-1}) \in G$ and $c \in \GL(4, q)$}
\li     $h := \proc{Random}(G)$
\li     $B := h^c$
\li     \If $u^B \neq u^{-1}$ \Then \label{line:B_u_compatible}
\li     $A := \diag(x y, x, x^{-1}, x^{-1} y^{-1}) \in \Mat_4(\F_q(x, y))$
\li     $\id{eqns} := \set{\Tr(AB)=0, \Tr(AB)^t=0}$ \label{line:alg_eqns}
\li     \Comment{$\id{eqns}$ determines $2$ polynomial equations in $x$ and $y$ (with $y$ used for $x^t$)}
\li     \If $\id{eqns}$ has a root $(r, r^t)$ for $(x, y)$ where $r \in \F_q^{\times}$ \Then \label{line:eqn_solve}
\li       $\hat{f} :=\diag(r^{t+1},r,r^{-1},r^{-t-1})B$
\li       \If $\abs{\hat{f}} = 4$ \Then  \label{line:a_order}
\li         $i := \proc{DiscreteLog}(\lambda, r)$
\li         \Comment{This will return $-1$ if $r \notin \gen{\lambda}$}
\li         \If $i \ge 0$ \Then \label{line:discrete_log_test}
\li             \Comment{Now $\lambda^i = r$}
\li             $f := g^i h$
\li             \Comment{Now $f$ is an SLP in $X$ and $\hat{f} = f^c$}
\li                 \Return $f$
                \End
\zi             \kw{end}
            \End
\zi         \kw{end}
        \End
\zi     \kw{end}
        \End
\zi     \kw{end}
        \End
\zi     \kw{end}
\li \Return \const{fail}
\end{codebox}

\begin{lem}\label{CanBeInStdSzq}
With $u$ and $B$ as in Algorithm \ref{alg:new_sz_trick}, there is a diagonal matrix
$\delta\in\GL(4,q)$ such that $u^{\delta}=u$ and $B^{\delta}$ lie in $\Sigma$.
\end{lem}
\begin{proof}
Let $\Delta$ be the subgroup of diagonal matrices of $\GL(4,q)$.

We observe that $\lambda^{t+1}, \lambda, \lambda^{-1}, \lambda^{-t-1}$ are distinct elements of
$\F_q^{\times}$, since $\lambda\ne 1$ (enforced by $g\ne 1$). For example, if $\lambda^{t+1}=\lambda^{-1}$,
we obtain $\lambda^{t+2}=1$, and raising to the power of $1-(t/2)$ gives us $\lambda=1$.
Thus $\cen_{\GL(4,q)}(u)=\Delta$.

The normaliser of $\gen{u}$ in $\GL(4,q)$ must permute the eigenspaces of $u$,
and thus any element therein must raise $u$ to one of the powers $\pm 1$, $\pm (t+1)$.
We observe that $T$ inverts $u$, while $u^{\pm(t+1)}$ has eigenvalues $\lambda^{\pm(t+1)}$,
$\lambda^{\pm(t+1)^2}$. Therefore, if $u^{\pm(t+1)}$ is conjugate in $\GL(4,q)$ to $u$ or $u^{-1}$,
then $\lambda^{(t+1)^2}=\lambda$ or $\lambda^{-1}$.
But $(t+1)^2=2q+2t+1\equiv 2t+3\ \mathrm{mod}\ q-1$,
and so we get $\lambda^{2t+3}=\lambda$ or $\lambda^{-1}$.
Now $2t+2$ and $2t+4$ are invertible powers when applied to $\F_q^{\times}$, and so $\lambda=1$.
This contradiction implies that $u$ is not conjugate in $\GL(4,q)$ to $u^{\pm(t+1)}$.
Thus $\nrm_{\GL(4,q)}(\gen{u})=\Delta\gen{T}=\gen{T}\Delta$.

Now it is well-known that all copies of $\Sz(q)$ in $\GL(4,q)$ are conjugate therein.
Notice also that $(u,B)=(g,h)^c$.
Thus $\gen{u,B}^{\gamma_1}\le \Sigma$ for some $\gamma_1\in\GL(4,q)$.
Moreover, there is $\gamma_2\in\Sigma$ such that $(u^{\gamma_1})^{\gamma_2}$
is a diagonal matrix $\mathrm{diag}(\mu^{t+1},\mu,\mu^{-1},\mu^{-(t+1)})$.
Naturally, the eigenvalues of this diagonal matrix coincide with the eigenvalues of $u$,
and the argument of the previous paragraph now shows that $\mu=\lambda^{\pm 1}$.
Therefore $u^{\gamma_1\gamma_2}=u^{\pm1}$. So $\gamma_1\gamma_2$ normalises $\gen{u}$,
and so $\gamma_1\gamma_2=\delta\zeta$, where $\delta\in\Delta$ and $\zeta=\zeta^{-1}=T^j$
for some $j\in \{0,1\}$.
Thus $\gen{u,B}^{\delta}=(\gen{u,B}^{\gamma_1})^{\gamma_2\zeta}
\le \Sigma^{\gamma_2\zeta}=\Sigma$, since $\gamma_2,\zeta\in\Sigma$, and this
establishes our claim.
\end{proof}

The next theorem asserts that the equations we have to solve in
Algorithm \ref{alg:new_sz_trick} have few solutions (with one
exception) and also provides a way of reducing the problem of their
solution to finding roots of a univariate polynomial of degree at most
4 over $\F_q$. To summarise the next theorem, let the diagonal entries of $B$ be $a,b,c,d$. Line \ref{line:eqn_solve} is then done as follows:
\begin{itemize}
\item Line \ref{line:B_u_compatible} excludes the case $a = b = c = d = 0$.
\item If exactly three of $a,b,c,d$ are $0$, then the equations at line \ref{line:alg_eqns} have no solutions of the required type.
\item If $a=b=0$ and $c,d \neq 0$ then $y = d/c$.
\item If $c=d=0$ and $a,b \neq 0$ then $y = b/a$.
\item In all other cases, equation \eqref{HEq3} is a non-degenerate univariate polynomial equation in $y$, and we consider all non-zero roots of this. The roots can be found using \cite[Algorithm 14.10]{VonzurGathen03}.
\end{itemize}
For each value of $y$ obtained above, we set $x = y^{t/2}$ and check if equation \eqref{HEq1} holds. This provides the values $r$ at line \ref{line:eqn_solve}. (There is no need to check equation \eqref{HEq2} since it will automatically hold if \eqref{HEq1} does.)

\begin{thm} \label{thm:elt_4}
The equations at line \ref{line:eqn_solve} in Algorithm \ref{alg:new_sz_trick} determine at most four solutions for $x$,
except when $B$ is an involution that inverts $u$.
\end{thm}
\begin{proof}
In the general case, we establish that $y$ is the root of a non-zero univariate
polynomial of degree at most 4. Exceptional cases are dealt with explicitly.

The analysis of one of the cases becomes simpler if $B$ is in our standard copy $\Sigma$.
By Lemma~\ref{CanBeInStdSzq}, there is a diagonal matrix $\delta$ such that $B^{\delta}$
is in $\Sigma$. Note that $A^{\delta}=A$, and so $\Tr(AB^{\delta})=
\Tr(A^{\delta}B^{\delta})=\Tr((AB)^{\delta})=\Tr(AB)$. In fact, $B$ and $B^{\delta}$
have the same diagonal entries. The upshot is that $\Tr(AB')=\Tr(AB')^t=0$ gives
the same equations for both $B'=B$ and $B'=B^{\delta}$.
From now on, we presume that $B$ (the old value of $B^{\delta})$ has been standardised
to lie in $\Sigma$.

We have $\Tr(AB) = axy+bx+cx^{-1}+dx^{-1}y^{-1}$, where $a=B_{11}$, $b=B_{22}$,
$c=B_{33}$ and $d=B_{44}$. Multiplying through by $xy$ and also considering $\Tr(AB)^t$ gives us:
\begin{eqnarray}
ax^2y^2+bx^2y+cy+d=0, && \label{HEq1} \\
a^tx^4y^2+b^tx^2y^2+c^tx^2+d^t =0. && \label{HEq2} 
\end{eqnarray}
We now try to solve these equations. Note that $x$ and $y$ determine each other via $y=x^t$ and $x=y^{t/2}$.
We now let
$$
p_1=(aa^ty^3+a^tby^2)x^2+ab^ty^3+bb^ty^2+a^tcy^2+ac^ty+a^tdy+bc^t
$$
and $p_2=a^2y^3+b^2y$. Evaluating $p_1\times\mbox{\eqref{HEq1}}+p_2\times\mbox{\eqref{HEq2}}$ gives us
\begin{equation}\label{HEq3}
\begin{array}{c}
ab^tcy^4 + (a^2d^t + ab^td + a^tc^2 + bb^tc)y^3 + (acc^t + bb^td)y^2 \\
+ (ac^td + a^td^2 + b^2d^t + bcc^t)y + bc^td = 0.
\end{array}
\end{equation}
(The values of $p_1$ and $p_2$ may be obtained using polynomial division to divide
$\Tr(AB)^tx^ty^t$ by $\Tr(AB)xy$, treating $x^2$ as the variable and everything else as a constant.)

Clearly, \eqref{HEq3} reduces to the trivial equation $0=0$ whenever $a=b=0$ or $c=d=0$.
We now show that it is non-trivial in all other cases, in which case $y$ is a root of a
non-trivial polynomial equation of degree at most $4$ (over $\F_q$), and thus can have at most four values.
(Note that a solution to \eqref{HEq3} does not necessarily yield a solution to
the original equations \eqref{HEq1} and \eqref{HEq2}.)
By inspecting the coefficients of $y^4$ and $1$ in \eqref{HEq3}, we see that this
equation can only be trivial if $ab^tc=bc^td=0$, putting us in one of the three cases
$a=d=0$, $b=0$ or $c=0$, which we now analyse separately below.

M-I. If $a=d=0$, \eqref{HEq3} reduces to $bb^tcy^3 + bcc^ty=0$ (so $bb^tcy^2+bcc^t=0$).
If $b=0$ or $c=0$, then we are in one of the excluded cases $a=b=0$ or $c=d=0$, and if
$b,c\ne 0$ we get $b^ty^2+c^t=0$, with unique solution $y=(c/b)^{t/2}$.

M-II. If $b=0$, \eqref{HEq3} reduces to $y(a^2d^ty^2 + a^tc^2y^2 + acc^ty + ac^td + a^td^2)=0$.
Because of the excluded case $a=b=0$, we may assume that $a\ne 0$. So for this equation
to be identically zero, $c$ must be $0$ (coefficient of $y^2$), giving us
$y(a^2d^ty^2 + a^td^2)=0$, and thus forcing $a^2d^t=a^td^2=0$. Since $a\ne 0$ this
forces $d^2=0$, whence $d=0$. But $c=d=0$ is also an excluded case.

M-III. If $c=0$, \eqref{HEq3} reduces to $y(a^2d^ty^2 + ab^tdy^2 + bb^tdy + a^td^2 + b^2d^t)=0$.
We have $d\ne 0$, since we are excluding $c=d=0$, so for a degenerate equation we need $b=0$
(from the coefficient of $y^2$). This also implies that we need $a^2d^t=a^td^2=0$.
So $a=0$ (as $d\ne 0$), putting us in another excluded case, namely $a=b=0$.

We now deal with the cases when \eqref{HEq3} is identically zero, namely $a=b=0$ and $c=d=0$.
With the exception of the case $a=b=c=d=0$, the original equations \eqref{HEq1} and \eqref{HEq2}
are not identically zero.

S-I. $a=b=0$, $c,d\ne 0$. In this case $y=d/c$, from \eqref{HEq1}.

S-II. $c=d=0$, $a,b\ne 0$. In this case $y=b/a$, from \eqref{HEq1}, as $x^2y\ne 0$.

S-III. Exactly three of $a,b,c,d$ are $0$. Now \eqref{HEq1} reduces
to one of $ax^2y^2=0$, $bx^2y=0$, $cy=0$ or $d=0$, which have no solutions, since we want $x,y\ne 0$.

S-IV. $a=b=c=d=0$. This does occur. The subspace generated by the first row of $B$ is a point of $\OV$,
so it is $P_{\infty}$ or $(e^2e^t+ef+f^t:f:e:1)$
for some $e,f\in \F_q$. The requirement that $B_{11}=0$ obviously rules out the
$P_{\infty}$ case, and imposes the requirement that $e^2e^t+ef+f^t=0$ in the other
case. This actually forces $e=f=0$.
So $B=M^{\prime}(k)U(e_1,f_1)T$ for some $k,e_1,f_1\in \F_q$ (with $k\ne 0$).
This element has $B_{11}=B_{22}=0$, $B_{33}=e_1^t/k$ and $B_{44}=(e_1^2e_1^t+e_1f_1+f_1^t)/(kk^t)$,
which imposes the relations $e_1=f_1=0$. So $B=M^{\prime}(k)T$ is an element of order $2$ in the
copy of $\dih{2(q-1)}$ that normalises the $\cyc{q-1}$ containing $u$. 

(We solve $e^2e^t+ef+f^t=0$ as follows. Taking $t$-th powers gives us $e^2e^{2t}+e^tf^t+f^2=0$;
take away $e^t$ times the original equation to leave us with $f^2+ee^tf=0$. This gives us
$f=0$ or $f=ee^t$, and substituting these into the original equation gives us $e^2e^t=0$ (so $e=0$)
in the case $f=0$ or $e^2e^t+eee^t+(ee^t)^t=0$ (so $e^te^2=0$) in the case when $f=ee^t$.
In either case, the only solution is $e=f=0$.
Note that $(z^{t+1})^{t-1}=z^{t^2-1}=z^{2-1}=z$ and $z^{t+2}=(z^{t+1})^t$ for all $z\in \mathbb{F}_q$,
so that $t+1$ and $t+2$ are both invertible powers when applied to $\F_q$.)
\end{proof}

Using the algorithms of \cite{MR3283836}, we can usually construct a
copy of $\Sz(q)$ as a subgroup of $\GL(4,q)$. However, the above
algorithm, lemma, theorem and proofs hold with obvious minor
modifications even when $G$ is a group isomorphic to $\Sz(q)$ embedded
in $\GL(4,F)$ for any field $F$ (including infinite ones) containing
$\F_q$.  Furthermore, the diagonal entries of the semi-standardised
matrix $B$ must lie in $\F_q$, even though the off-diagonal entries
need not do so.

\begin{cl} \label{cl:4_elts}
Let $G \cong \Sz(q)$, where $q = 2^{2m+1}$ for some $m > 0$, and let $g \in G$
have order $q - 1$.
At least $7/32$ of the cosets of $\gen{g}$ have at least one element of order $4$.
\end{cl}
\begin{proof}
In total in $G$, there are $q(q-1)(q^2 + 1)$ elements of order $4$, and $q^2(q^2 + 1)$
cosets of $\gen{g}$. By Theorem \ref{thm:elt_4} there are at most four elements of order $4$
in each coset. Hence there are at least $q(q-1)(q^2 + 1)/4$ cosets containing an element of
order $4$, and the proportion of such cosets is $1/4 - 1/(4q) \ge 7/32$, since $q \ge 8$.
\end{proof}

\begin{thm} \label{thm:alg_prob}
Algorithm~\ref{alg:new_sz_trick} is a Las Vegas algorithm, with success
probability at least
$7/(384 \log\log(q)) - \OR{1/q}$.
When the algorithm succeeds it returns an element as an\/ $\SLP$ in $X$.
\end{thm}
\begin{proof}
It is clear that $g$ and $h$ are constructed as $\SLP$s, hence $f$ will have an $\SLP$ as well.

There are $(q - 2)q^2(q^2 + 1)/2$ non-trivial elements of order
dividing $q-1$, hence the test at line \ref{line:finding_g} succeeds
with probability $(q/2-1)/(q-1)$, which is at least $3/7$.

From now on, we can (and shall) assume that $g$ is uniformly distributed among the non-trivial
elements of $G$ having order dividing $q-1$. Since $\Cent_G(G)\cong\cyc{q-1}$, any two distinct
subgroups of $G$ isomorphic to $\cyc{q-1}$ have trivial intersection.
So we can, for the purpose of this proof, choose $g$ by randomly picking a copy of 
$\cyc{q-1}$, and let $g$ be a random non-trivial element therein, in both cases using a uniform
distribution. We define $C:=\Cent_G(g)\cong\cyc{q-1}$, and $D:=\nrm_G(\gen{g})=\nrm_G(C)\cong \dih{2(q-1)}$.
The \mbox{(right-)cosets} of $C$ in $G$ partition $G$ into sets all of the same size.
So we can choose $h$ by first choosing a random coset $S$ (of $C$), and letting $h$ be a random
element of $S$, again with all distributions uniform.

There are just $q-1$ elements that invert $g$, namely those elements that lie in the coset
$D\setminus C$ of $C$, and all of them are involutions,
so the test at line \ref{line:B_u_compatible} succeeds with probability $1 - \OR{1/q}$.
(Note that the pairs $(u,B)$ and $(g,h)$ are conjugate in $\GL(4,q)$.)
So now we can assume that $S$ is uniformly distributed among the cosets of $C$ other than $D\setminus C$. We use Theorem \ref{thm:elt_4} to reduce the equations on line \ref{line:eqn_solve} to a univariate polynomial, and we can then compute its roots using \cite[Algorithm 14.10]{VonzurGathen03}. Using $3\log_2(q)$ applications of \cite[Algorithm 14.8]{VonzurGathen03} in the latter, the probability of success is $1-\OR{1/q}$.
The solutions of these equations are in one-to-one correspondence with
the elements $h_1,\ldots,h_s$ of $S$ having order lying in $\{1,2,4\}$.
By Corollary \ref{cl:4_elts}, these equations have a solution with probability
at least $7/32$, and Algorithm~\ref{alg:new_sz_trick} produces an element
$f=c\hat{f}c^{-1}$ which is one of $h_1,\ldots,h_s$.


Unfortunately, given $g$, the elements $f$ produced are not uniformly distributed
among the elements of orders $1,2,4$ not lying in $D\setminus C$.
For if $h_1,\ldots,h_s$ do occur with the same probability, then
$f$ has a probability of $[s(q^4-q+1)]^{-1}$,
and $s$ can vary, depending on the coset $S$.
But $G$ has just $(q^2+1)(q-1)$ involutions, and $1$ identity,
while Corollary \ref{cl:4_elts} implies that at least $7(q^4+q^2-1)/32$ cosets $S(\ne D\setminus C)$
will have an element of order $4$.
Thus $h_1,\ldots,h_s$ will all have order $4$ with probability $1-\OR{1/q}$,
and so the test at line \ref{line:a_order} succeeds with probability $1 - \OR{1/q}$.

If $\abs{g} = q-1$ then the test at line~\ref{line:discrete_log_test}
always succeeds, and by \cite[Lemma 2.5]{baarnhielm05} this happens
with probability at least $1/(12 \log\log(q))$.

(Note that in general, if we choose a solution at line~\ref{line:eqn_solve}
uniformly at random, then $f$ will be a random element of $h_1,\ldots,h_s$,
chosen uniformly, depending only on the coset $S$, and independent of $h$.
Then the discrete log test at line~\ref{line:discrete_log_test} succeeds
if and only if $h\in \gen{g}f$, which happens with probability
$\abs{\gen{g}f}/\abs{S} = \abs{g}/(q-1)$.)

We see that the success probabilities of lines \ref{line:finding_g},
\ref{line:B_u_compatible}, \ref{line:eqn_solve}, \ref{line:a_order} and
\ref{line:discrete_log_test} combine to at least $7/(384 \log\log(q)) - \OR{1/q}$.
\end{proof}

We remark that successful runs of Algorithm~\ref{alg:new_sz_trick} produce
elements of order $4$ in $G$ uniformly at random, provided that the
solutions found at line~\ref{line:eqn_solve} are chosen uniformly at random;
the fact that $g$ is random smooths out the outputs obtained when $g$ is fixed.

\begin{thm} \label{thm:alg_time}
  Given an oracle for the discrete logarithm problem in $\F_q$ and a
  random element oracle, Algorithm \ref{alg:new_sz_trick} has expected time
  complexity $\OR{\log(q)}$ field operations.
\end{thm}
\begin{proof}
  All matrix arithmetic can be done in $\OR{1}$ field operations,
  except exponentiation which can be done in $\OR{\log(q)}$ field
  operations, for example by \cite[Lemma 10.1]{classical_recognise}.
  Computation of characteristic polynomial and eigenspaces also
  requires $\OR{1}$ field operations. 

  By \cite{crlg95}, line \ref{line:finding_g} has expected time complexity
  $\OR{\log(q)}$ field operations.

  At lines \ref{line:diagonalise} and \ref{line:eqn_solve} we find the
  roots of a polynomial over $\F_q$ of degree at most $4$. The expected time
  complexity is $\OR{\log(q)}$ field operations, by \cite[Corollary
  14.11]{VonzurGathen03}.
\end{proof}

In practice, when implementing Algorithm~\ref{alg:new_sz_trick}, in order to avoid
unnecessary work, and excessive calls to \proc{Random}(G), on rerunning the
algorithm we only execute those steps that are necessary.

Thus if Algorithm~\ref{alg:new_sz_trick} fails because $g$ has the
wrong order or \proc{DiscreteLog} does not succeed, we should rerun the algorithm
with a new value of $g$ (we can retain $h$ if we choose). In other cases
of failure, we should retain $g$ and use a new value of $h$.
If the algorithm reaches the \proc{DiscreteLog} stage, it will succeed there
with probability $\abs{g}/(q-1)$. Taking account of this small probability of success is
awkward, and suppressed analysis shows that the success probability for this
stage does not exceed $1.03\phi(q-1)/(q-2)$.
This means that the expected number of executions of \proc{DiscreteLog} before a
successful run of Algorithm~\ref{alg:new_sz_trick} is at least
$0.97(q-2)/\phi(q-1)$. Formally, this quantity has magnitude $\OR{\log\log(q)}$.
In practice, while it is the case that $n/\phi(n)>\log\log(n)$
for infinitely many values of $n$,
we do not know the behaviour when $n$ is $q-1$ for $q=2^{2m+1}$.

We can also modify Algorithm~\ref{alg:new_sz_trick} and introduce exactly one call
to integer factorisation (to factor $q-1$), hence only having to call \proc{DiscreteLog} exactly once.
In this case, line~\ref{line:order_g} now tests whether $g$ has order exactly $q-1$.
This takes time $\OR{\log q(\log t+1)}$ field operations by \cite{crlg95}, where $t$ is
the number of distinct prime factors of $q-1$, and so time $\OR{\log(q)\log\log(q)}$ in the worst case.
We now run lines~\ref{line:producing_g} and \ref{line:order_g} an expected $2(q-1)/\phi(q-1)$
times until we have produced an element $g$ of order $q-1$, and do not run this part
of the algorithm again. We then execute lines 3 and 4, and repeat the rest of the
algorithm (lines 5--21) until success occurs.
The tests on lines \ref{line:B_u_compatible}, \ref{line:eqn_solve} and \ref{line:a_order}
succeed with the same probabilities as before,
but now, when executed, the discrete log test on line~\ref{line:discrete_log_test} always succeeds
(and in this case the entire algorithm succeeds).
In this variant the success probability of lines 5--21 is at least $7/32-\OR{1/q}$.
Hence, the expected time complexity until this variant of the algorithm succeeds
(ignoring the single call to \proc{DiscreteLog})
is $\OR{\log(q) (\log\log(q))^2}$ field operations, which is dominated by finding $g$.

\begin{thm} \label{thm_bray_stab_trick} Given a random element oracle
  for subgroups of $\GL(4,q)$, there exists a Las Vegas algorithm
  that, given $G = \gen{X} \le \GL(4, q)$ such that $G \cong
  \Sz(q)$, and an element $f \in G$ of order $4$, constructs $Y=\{f,h\}
  \subseteq \GL(4, q)$ such that $\O2(G_P) < \gen{Y} \le G_P$, where $P$ is the
  unique point of the $G$-ovoid fixed by $f$. The algorithm has expected time complexity
  $\OR{\log (q)\log\log(q)}$ field operations.
  (If $G=\Sigma^g$, where $\Sigma$ is our standard $\Sz(q)$, then the $G$-ovoid is $\OV^{g}$.)
\end{thm}
\begin{proof}
The algorithm proceeds as follows:
\def\notea{note (a)}
\def\noteb{note (b)}
\def\notec{note (c)}
\begin{enumerate}
\item Use the algorithm of \cite{bray00} to construct $g \in
  \Cent_G(f^2) \cong \F_q {.} \F_q$.
  Thus we pick a random $c\in G$ such that $[f^2,c]$ has odd order $2k+1$ for some $k\in\N$,
  and we calculate $g:=c[f^2,c]^k$ (see \notea\ below about this).
  If $[f^2,c]$ has even order then start again.

  If $\abs{g} = 4$ then let $j := g^2$, otherwise $j := g$.
  Repeat this step until $j\notin\{1,f^2\}$.
  This requires expected time $\OR{\log(q)}$ field operations, and the element $j$ produced is
  uniformly distributed among the set $I$ of $q-2$ involutions of $G_P\setminus\{f^2\}$.
  See \noteb\ below for further details about this step.
\item Construct random $c \in G$ such that $j^c \notin\Cent_G(f^2)$ (see \notec\ below). An equivalent requirement is that $c\notin G_P$, and so this step requires $\OR{1}$ field operations. The product of two involutions has odd order if they lie in distinct point stabilisers, so $\abs{f^2 j^c} = 2k+1$ for some $k \in \N$.
\item Let $h = c (f^2 j^c)^k$ (see \notea\ below). This requires $\OR{\log (q)}$ field operations.
      Then $j^h = f^2$, so $h$ must fix the point fixed by $f^2$ and $j$,
      and hence $\gen{f, h}$ lies in a point stabiliser.
\item Now $\abs{h}=2$ or $4$ or $\abs{h}\mid q-1$, since $\gen{f,h}$ lies in a point stabiliser,
  and $h \notin \Cent_G(f^2)$ now forces $\abs{h}\ne 1,2,4$.

  For each prime divisor $p$ of $2m+1$, set $a_p=2^{(2m+1)/p}-1$, and \mbox{verify} that $h^{a_p}\ne 1$.
  Each test takes time $\OR{\log(q)}$ field operations, and we have to do at
  most $\OR{\log\log(q)}$ of them, for an overall time complexity of $\OR{\log(q)\log\log(q)}$.
  (We have no divide-and-conquer strategy of \cite{crlg95} to speed things up.)
  Otherwise (if $h^{a_p}=1$ for some $p$) return to Step 1.
\item Set $Y = \set{f, h}$, and return $Y$. Note that $\gen{Y}$ is the whole of the
  point stabiliser $G_P$ if and only if $\abs{h}=q-1$.
\end{enumerate}
Some notes on aspects of the above algorithm are given below.
\begin{enumerate}
\item[(a)] 
In order to prove that $g\in G\cong \Sz(q)$ has odd order, it suffices to show
that $g=1$ or $g^4\ne 1$. And if $g$ has odd order $2k+1$, but we only know an
odd multiple $(2k+1)(2l+1)=2(2kl+k+l)+1$ of this order, then we can
still compute $g^k$, since $g^{2kl + k + l} = g^{l (2k + 1) + k} = g^k$.
In our case, $(q^2+1)(q-1)$ is always a multiple of the order of any odd order element.
\item[(b)] 
The probability of success at each iteration is $1-\OR{1/q}$. Note that $[f^2,c]$ has odd order whenever
  $c\notin G_P$, and also when $c\in \O2(G_P)$, and that the possible $g$ obtained this way are uniformly
  distributed throughout all of $\Cent_G(f^2)=\O2(G_P)$. Among the $q^2$ possibilities for $g$,
  precisely $2+q$ of them produce a forbidden value of $j$. The other $q^2-q-2=(q-2)(q+1)$ of them
  produce a $j$ uniformly distributed among the $q-2$ elements of $I$.

  (If $[f^2,c]$ has even order, necessarily 2, then we can set $j:=g=[f^2,c]$
  (not quite in line with \cite{bray00}). The element $j$ depends solely on the coset
  $\O2(G_P)c$, always lies in $I$, and has uniform distribution therein.)
\item[(c)] 
In Step 2, letting $c$ run through $X$ until success occurs, as it must, may be the best way
to execute this step, especially if $X$ is small, which will typically be the case.
In theory, $f$ is independent of $X$, and so the probability of success
for each choice of $c\in X$ is $1-\OR{1/q}$.

Note that restricting $c$ to lie in $X$ does not change the success
probability of subsequent steps, notably Step 4. This is because all
elements conjugating $j$ to $f^2$ lie in the coset $\O2(G_P)h$, they
all have the same order, so the order
of $h$ depends only on the pair $(f^2, j)$, which is fully
constructed by the end of Step 1. Hence it is not wise to use $c\in X$
when executing Step 1, as this determines $j$, and Step 4 could fail
for all choices of $c\in X$.
\item[(d)] 
If we have used the integer factorisation oracle earlier (to factor $q-1$),
then it seems neater to generate the whole of $G_P$ here, and in that case,
at Step 4, we test if $h$ has order $q-1$. This test takes time
$\OR{\log(q)\log\log(q)}$ field operations, 
and succeeds with probability $\phi(q - 1) / (q - 2)$.
This variant of the algorithm has expected time complexity
$\OR{\log (q)(\log\log(q))^2}$ field operations.
\end{enumerate}
To finish off the proof, we must calculate the probability that Step 4 succeeds,
and show that correct output is returned. But the possible elements $h$ correspond to
non-trivial cosets of $\O2(G_P)$ in $G_P$ (any element of $\O2(G_P)h$ conjugates
$j$ to $f^2$), Moreover, $h$ and $\O2(G_P)h$ have the same order, and the possible
elements $h$ are uniformly distributed, and so the order distribution of $h$ is
the same as that for the non-trivial elements of a $\cyc{q-1}$.
Each order test carried out in Step 4 excludes at most $\sqrt[3]{q}-2$ of these, and
there are $\OR{\log(q)}$ such tests. Thus Step 4 succeeds with probability $1-\OR{1/\sqrt{q}}$.
Now $h$ acts on $\O2(G_P)/\Phi(\O2(G_P))\cong (\F_q,+)$ in a manner corresponding
to the $\F_2$-action of $\lambda$ (or $\lambda^t$) for some $\lambda\in\F_q\setminus\set{0,1}$,
and the $\F_2\!\gen{h}$-module $\O2(G_P)/\Phi(\O2(G_P))$ is irreducible
if and only if $\lambda$ does not belong to a proper subfield.
Our order restrictions on $h$ rule out that case, giving an irreducible action.
Therefore $\O2(G_P)\le \gen{f,h}$, since $f\in \O2(G_P)\setminus\Phi(\O2(G_p))$.
\end{proof}

\section{Constructive recognition}

The new constructive recognition algorithm for the natural
representation is given as the following two results.
The first of these describes how to conjugate an arbitrary
copy $\gen{X}$ of $\Sz(q)$ to the standard copy, and the second uses
the first to provide deterministic constructive membership testing of an
element of $\GL(4,q)$ inside $\gen{X}$.

\begin{thm} \label{main_theorem} Given a random element oracle for
  subgroups of\/ $\GL(4, q)$ and an oracle for the discrete logarithm
  problem in $\F_q$, there exists a Las Vegas algorithm that, for each
  conjugate $\gen{X}$ of\/ $\Sigma$, constructs $g \in \GL(4, q)$ such
  that $\gen{X}^g = \Sigma$ and rewriting generators $\alpha_1,h_1,\gamma$ of
  $\gen{X}$ with respect to $g$ as\/ $\SLP$s in $X$.

  The algorithm has expected time complexity\/ $\OR{\log(q)\log\log(q)}$
  field operations. The discrete logarithm oracle is only needed in the
  initial phase, in order to obtain an element of order\/ $4$, where it is
  used, at worst, $\OR{\log\log(q)}$ times.
\end{thm}
\begin{proof}
Let $G = \gen{X}$.
The algorithm proceeds as follows.
\begin{enumerate}
\item Use Algorithm \ref{alg:new_sz_trick} to construct an element $\alpha_1$ of order $4$.
By Theorem \ref{thm:alg_prob}, the expected number of invocations of the algorithm is
$\OR{\log\log(q)}$, so by Theorem \ref{thm:alg_time}, this step requires expected time
$\OR{\log(q)\log\log(q)}$ field operations.
\item Use Theorem \ref{thm_bray_stab_trick} to construct a set of matrices $Y_1=\set{\alpha_1,h_1}$ such that
  $\O2(G_P) < \gen{Y_1} \le G_P$ for some $P$ in our $G$-ovoid. Use $\alpha_1$ for this.
  This requires expected time $\OR{\log(q)\log\log(q)}$ field operations.
  Theorem~\ref{thm_bray_stab_trick}, Step~4 forces $h_1$ to have the required order.
\item There are just three non-trivial proper submodules of $\gen{Y_1}$ (and even of $\gen{\alpha_1}$),
namely $V_1^P<V_2^P<V_3^P$, where $\dim V_i^P=i$ for $i=1,2,3$.
For each $i$, we can obtain $V_i^P$ as the nullspace of $(\alpha_1-1)^i$.
This requires expected time $\OR{1}$ field operations.
\item Choose random $\beta \in G$ that does not fix $P$ (the subspace $V_1^P$ constructed in Step 3),
and let $\gamma = (\alpha_1^2)^{\beta}$ and $Q = P\gamma$.
Then $Q \neq P$ and hence $\gen{Y_1}^{\gamma}=\gen{Y_2} \le G_Q$,
where $Y_2=\set{\alpha_2,h_2}=\set{\alpha_1^{\gamma},h_1^{\gamma}}$.
Also $Y_2^{\gamma}=Y_1$, since $\gamma$ is an involution.
This requires expected time $\OR{1}$ field operations.
(It is probably sensible to take $\beta\in X$ here.)

Note that $G$ is generated by any two of its distinct Sylow $2$-subgroups.
Thus $G=\gen{\O2(G_p),\O2(G_Q)}=\gen{\alpha_1,h_1,\alpha_2,h_2}=\gen{\alpha_1,h_1,\gamma}$.
\item Use the method of Step 3 to construct the non-trivial proper submodules
$V_1^Q<V_2^Q<V_3^Q$ of $\gen{Y_2}$ (so $V_i^Q$ is the nullspace of $(\alpha_2-1)^i$).
\item Define $U_1=V_1^P$, $U_2=V_2^P\cap V_3^Q$, $U_3=V_3^P\cap V_2^Q$, $U_4=V_1^Q$.
For $i=1,2,3,4$, we have $\dim U_i=1$, and also $\gamma$ swaps $U_1$ with $U_4$ and $U_2$ with $U_3$.
We choose nonzero $u_1\in U_1$ and $u_2\in U_2$,
and define $u_3=u_2\gamma\in U_3$ and $u_4=u_1\gamma\in U_4$.
Let $k$ be the inverse of the matrix whose $i$-th row is $u_i$.
Then by the proof of \cite[Lemma 7.4]{baarnhielm05}, there is a diagonal matrix $d\in\GL(4,q)$
such that $(G^k)^d=\Sigma$, our standard copy of $\Sz(q)$. This requires expected time $\OR{1}$ field operations.
(We have corrected the definitions of $U_2$ and $U_3$ here.)
\item Let $J=\antidiag(1,1,1,1)$, and $d=\diag(d_1,d_2,d_3,d_4)$, where $(G^k)^d=\Sigma$.
Then $G^k$ preserves the form $dJd^{\mathsf{T}}=\antidiag(d_1d_4,d_2d_3,d_2d_3,d_1d_4)$,
a form that is unique up to scalars, since $G$ acts absolutely irreducibly on its natural module.
Since $G^k$ contains $\gamma^k=\antidiag(1,1,1,1)$, $\Sigma$ contains
$d^{-1}(\gamma^k)d=\antidiag(d_1^{-1}d_4,d_2^{-1}d_3,d_3^{-1}d_2,d_4^{-1}d_1)$, and so
this must be $M'(\kappa)T$ for some $\kappa\in\F_q^{\times}$, and conjugating this
by the diagonal matrix $M'(\sqrt{\kappa})\in\Sigma$ gives us $T(=J=\gamma^k)$.

Therefore, we may assume that $d$ centralises $\gamma^k$, and this forces $d=\diag(d_1,d_2,d_2,d_1)$.
But conjugating by scalars has no effect, and so we can take $d_1=1$. Therefore
$d=\diag(1,d_2,d_2,1)$ for some $d_2$, and the form preserved by $G^k$ is $K:=\antidiag(1,d_2^2,d_2^2,1)$.
The equations $hKh^{\mathsf T}=K$ for $h\in\set{\alpha_1^k,h_1^k}$
give us many linear equations for $d_2^2$, at least some of which are non-trivial.
(Note that $\gamma^k$ automatically preserves the form $K$.)
This requires $\OR{\log q}$ field operations (to square-root $d_2^2$).

Note that using $\gamma$ to partially standardise $G^k$ simplifies this step compared
to \cite[Lemma 7.3]{baarnhielm05}.
%
\item Let $g=kd$, where $k$ is as constructed in Step 6 and $d$ is as constructed in Step 7.
Then $G^g=\Sigma$, the standard copy of $\Sz(q)$.
It is clear from the proof of \cite[Lemmas 7.3 \& 7.4]{baarnhielm05} that $Pg = P_{\infty} = (1:0:0:0)$ and $Qg = P_0 = (0:0:0:1)$,
hence $\Sigma_{P\infty} \ge \gen{Y_1}^g$ and $\Sigma_{P_0} \ge \gen{Y_2}^g$.
\end{enumerate}
Most of the output criteria and complexity issues have been dealt with as we went along.
To finish off, we note that $\gamma^g=(\gamma^k)^d=T^d=T$.
\end{proof}

\begin{thm} \label{main_theorem2}
Let $X,\alpha_1,h_1,\gamma,g$ be as in Theorem~\ref{main_theorem}.
Thus $G=\gen{X}=\gen{\alpha_1,h_1,\gamma}\cong\Sz(q)$ is a subgroup of\/ $\GL(4,q)$
and $G^g=\Sigma$. Let $h\in \GL(4,F)$. Then:
\begin{enumerate}
\item We can determine whether $h\in G$ in $\OR{1}$ field operations.
\item Given some preprocessing that is independent of $h$, and given that $h\in G$,
we can construct $h$ as an\/ $\SLP$ in $\alpha_1,h_1,\gamma$ in time $\OR{\log(q)}$ field operations.
The\/ $\SLP$ has length\/ $\OR{\log(q)}$. Thus we also get $h$ as an\/ $\SLP$ in $X$.
The preprocessing, which only has to be done once for any $(X,\alpha_1,h_1,\gamma,g)$,
has complexity at most\/ $\OR{(\log(q))^2}$, the true value being dependent on
the complexity of matrix inversion.
\end{enumerate}
(We are counting $\log_2(q)$ bit operations as being $1$ field operation, as one
field operation over\/ $\F_q$ must take at least\/ $\log_2(q)$ bit operations.)
\end{thm}
\begin{proof}
%
We note that $h\in \GL(4,q)$ belongs to $G$ if and only if $h^g\in G^g=\Sigma$, which we solve as follows.
We make use of the partitioning from \eqref{sz_partition}.
\begin{enumerate}
\item Consider the first row of $h^g$. If it is $\mu(1,0,0,0)$ for some $\mu\ne 0$ let $k_1=\I_4$.
If it is $c(a^{t+2}+ab+b^t,b,a,1)$ for some $a,b,\mu$ with $\mu\ne 0$ let $k_1=(TU(a,b))^{-1}$.
In any other case, $h^g$ does not preserve $\OV$, so return \const{fail}.
In successful cases, let $k_0=(h^g)k_1^{-1}$.
\item Now the first row of $k_0$ has the form $\mu(1,0,0,0)$ for some $\mu\ne 0$.
So in the case when $h^g\in\Sigma$, we have $k_0\in\Sigma$, and so $k_0=M'(\lambda)U(c,d)$
for some $c,d,\lambda$ with $\lambda\ne 0$.

Look at the $(2,2)$-entry of $k_0$. This should be $\lambda$, so if it is $0$ then return \const{fail}.
In the other cases, define $k_3=M'(\lambda)$, and $k_2=k_3^{-1}k_0$.
If $k_2$ is $U(c,d)$ for some $c,d$ then we succeed, otherwise we return \const{fail}.
\item In the cases of success, we have also written $h^g$ in the form $k_3k_2k_1=M'(\lambda)U(c,d).1$
or $M'(\lambda)U(c,d)TU(a,b)$ for some $a,b,c,d,\lambda$ with $\lambda\ne 0$.
\end{enumerate}
All the above clearly requires just $\OR{1}$ field operations.
In the successful cases, we wish to write $h$ as an SLP in $X$.
First, we show how to write each of the elements $T$, $M'(\lambda)$ and $U(a,b)$
in terms of $f=\alpha_1^g$, $e=h_1^g$ and $z=\gamma^g$.
The easy one is that $T=z$. We also have
$M'(\lambda)=zU(0,\lambda^{1+t/2})zU(\lambda^{-t/2},\lambda^{-1-t/2})zU(\lambda^{t/2},0)$,
which just defers the problem. Last, but not least, we consider $U(a,b)$.
We have $f=U(a_1,b_1)$ and $e=M'(\mu)U(a_2,b_2)$ for some $a_1,b_1,a_2,b_2,\mu$ with $a_1,\mu\ne 0$.
Now $\O2(\gen{e,f})=\O2(G_{P_{\infty}})$ has generating set
$$
L:=\{\,f^{e^i}, (f^2)^{e^i}\mid 0\le i\le 2m\,\}.
$$
In terms of matrices
$$
f^{e^i} = U(\mu^{it}a_1,*)\quad\mbox{and}\quad
(f^2)^{e^i} = U(0,\mu^{(t+1)it}a_1^{t+1}) = U(0,\mu^{(t+2)i}a_1^{t+1}),
$$
where we have not calculated the starred entry.
By construction of $h_1$,
$\mu$ does not lie in a subfield of $\F_q$, and hence the elements $\mu^{it}a_1$
and $\mu^{(t+2)i}a_1^{t+1}$ form vector space bases for $\F_q$ over $\F_2$.
Now solve a linear system over $\F_2$ and calculate $n_0,\ldots,n_{2m}\in \set{0,1}$
such that $a=a_1(n_0+n_1\mu^t+\cdots +n_{2m}\mu^{2mt})$.
We set 
$$
j_1= \prod_{i = 0}^{2m} (f^{e^i})^{n_i} = f^{n_0}e^{-1}f^{n_1}e^{-1}\ldots e^{-1}f^{n_{2m}}e^{2m},
$$
and note that its matrix has form $U(a,*)$.
We have $U(a,b)=j_1U(0,\beta)$ for some $\beta$, and solve another linear system to obtain
$\beta=a_1^{t+1}(p_0+p_1\mu^{t+2}+\cdots +p_{2m}\mu^{2m(t+2)})$
for some $p_0,\ldots,p_{2m}\in\set{0,1}$.
This gives us 
$$
j_2= \prod_{i = 0}^{2m} ((f^2)^{e_i})^{p_i} =
f^{2p_0}e^{-1}f^{2p_1}e^{-1}\ldots e^{-1}f^{2p_{2m}}e^{2m}=U(0,\beta),
$$
and so $U(a,b)=j_1j_2$ writes $U(a,b)$ as a word or SLP of length $\OR{\log(q)}$ in $\set{e,f}$.

Note that the matrices, and their inverses, used in the linear system solving
do not depend on $h$, and hence we precompute them.
(This leads to $\Theta(\log(q))$ space complexity of the
algorithm, which may be unavoidable in any case.)
We then obtain $n_i$ and $p_i$ by multiplication with these inverse matrices,
which requires $\OR{(\log(q))^2}$ bit operations, and thus $\OR{\log(q)}$ field ($\F_q$) operations.
Therefore, writing $U(a, b)$ as an $\SLP$ in $e$ and $f$
also requires $\OR{\log(q)}$ field operations.
The precomputation requires us to invert two degree $\log_2(q)$ matrices over $\F_2$,
for which the classical algorithm uses $\OR{(\log(q))^3}$ bit operations and thus
$\OR{(\log(q))^2}$ field operations. (It is known that asymptotically faster matrix
inversion algorithms exist. It is not known whether it is possible for Gaussian
elimination to be asymptotically faster than matrix inversion.)

Having shown how to write each of the elements $T$, $M'(\lambda)$ and $U(a,b)$ as
SLPs in $\alpha_1^g$, $h_1^g$, $\gamma^g$, we can now easily obtain $h^g$ as
an SLP in $\alpha_1^g$, $h_1^g$, $\gamma^g$ since we have already noted that
$h^g=M'(\lambda)U(c,d)$ or $M'(\lambda)U(c,d)TU(a,b)$ for some $a,b,c,d,\lambda$.
The same SLP gives $h$ in terms of $\alpha_1$, $h_1$, $\gamma$, and since these
three elements have known SLPs in terms of $X$, so now does $h$.

As we have seen, writing $h$ as an $\SLP$ in $\set{\alpha_1,h_1,\gamma}$
requires at most $5$ invocations of the
above method that writes an element of $\mathcal{F}$ as an $\SLP$ in $\set{e,f}$.
Therefore this requires time complexity $\OR{\log(q)}$ field operations, and
produces an SLP for $h$ having length $\OR{\log(q)}$ in $\set{\alpha_1,h_1,\gamma}$.
\end{proof}

\section{Elements of order $4$}

Following Corollary \ref{cl:4_elts}, we conjecture that the actual
proportion of cosets of $\gen{g}$ possessing an element of order $4$
is
$$
\frac{5q^3-3q^2+14q-16}{8q(q^2+1)},
$$
a value we have checked for $q=2,8,32,128,512$. We associate the vector $(v_1,v_2,v_4)$ to a coset of
$\gen{g}$, where $v_i$ is the number of elements of order $i$ in that coset. We have proved that
there are just nine possibilities for this vector. (Theorem \ref{thm:elt_4} restricts the possible vectors, but some other arguments are needed too.)
The possible vectors, and the number of cosets (of $\gen{g}$) in $\Sz(q)$ that we conjecture have
this vector, are tabulated below.
\begin{center}
\begin{tabular}{rlrl}
vector & \#cosets & vector & \#cosets \\[2pt]
$(0,0,0)$ & $\frac{1}{8}(q-1)(3q^3+2q^2-8q+16)$ & $(0,1,0)$ & $\frac{1}{2}(q-1)q(q+2)$ \\[4pt]
$(0,0,1)$ & $\frac{1}{6}(q-1)q(2q^2+q+20)$ & $(0,1,2)$ & $\frac{1}{2}(q-1)q(q-2)$ \\[4pt]
$(0,0,2)$ & $\frac{1}{4}(q-1)q^2(q-2)$ & $(0,q-1,0)$ & $1$ \\[4pt]
$(0,0,3)$ & $\frac{1}{2}(q-1)q(q-2)$ & $(1,0,0)$ & $1$ \\[4pt]
$(0,0,4)$ & $\frac{1}{24}(q-1)q(q-2)(q-8)$ & &
\end{tabular}
\end{center}
Again, we have checked these values for $q=2,8,32,128,512$. For the case $q=2$ and vector $(0,1,0)$
one should sum the values listed for the cases $(0,1,0)$ and $(0,q-1,0)$.

\section{Implementation and performance}

An implementation of the algorithms described here is available in
$\MAGMA$ \cite{magma}, as part of the \textsc{CompositionTree} package
\cite{MR3283836}. The
implementation uses the existing $\MAGMA$ implementations of the
algorithms described in \cite{lg95}, \cite{crlg95} and \cite[Corollary
14.10]{VonzurGathen03}.

A benchmark of the algorithm in Theorem \ref{main_theorem}, for field
sizes $q = 2^{2m + 1}$, with $m = 1, \ldots, 100$, is given in
Figure~\ref{fig:recognition_benchmark}. For each field size, $100$ random
conjugates of $\Sz(q)$ were recognised, and the average running time
for each call, as well as the average time spent in discrete logarithm
calculations, is displayed. As expected, the running time is
completely dominated by the time to compute discrete logarithms,
and the two plots in Figure~\ref{fig:recognition_benchmark} are almost indistinguishable.

\begin{figure}[ht]
\includegraphics[scale=0.75]{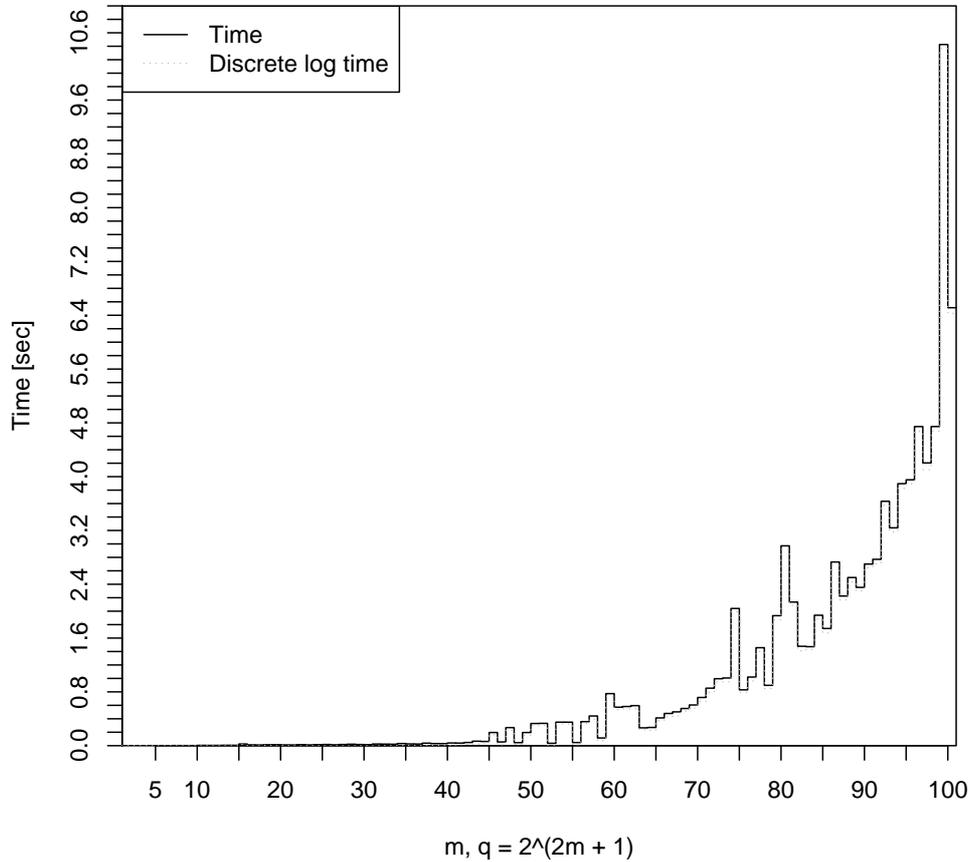}
\caption{Benchmark of recognition}
\label{fig:recognition_benchmark}
\end{figure}

The timings jump up and down due to the cost of discrete logarithms
and due to the way that $\MAGMA$ handles finite field computations: it
uses Zech logarithms for finite fields up to a certain size, and for
larger fields it tries to find a subfield smaller than this
size. Hence field arithmetic speed depends on the prime divisors of
$2m+1$.

We have also benchmarked the algorithm in Theorem \ref{main_theorem2}
in a similar way. For each field size, $\SLP$s of $100$ random
elements were calculated and the average running time for each call is
displayed in Figure \ref{fig:membership_benchmark}. The time to
precomputate the matrices used in the linear system solving is not
included in the running time.

\begin{figure}[ht]
\includegraphics[scale=0.75]{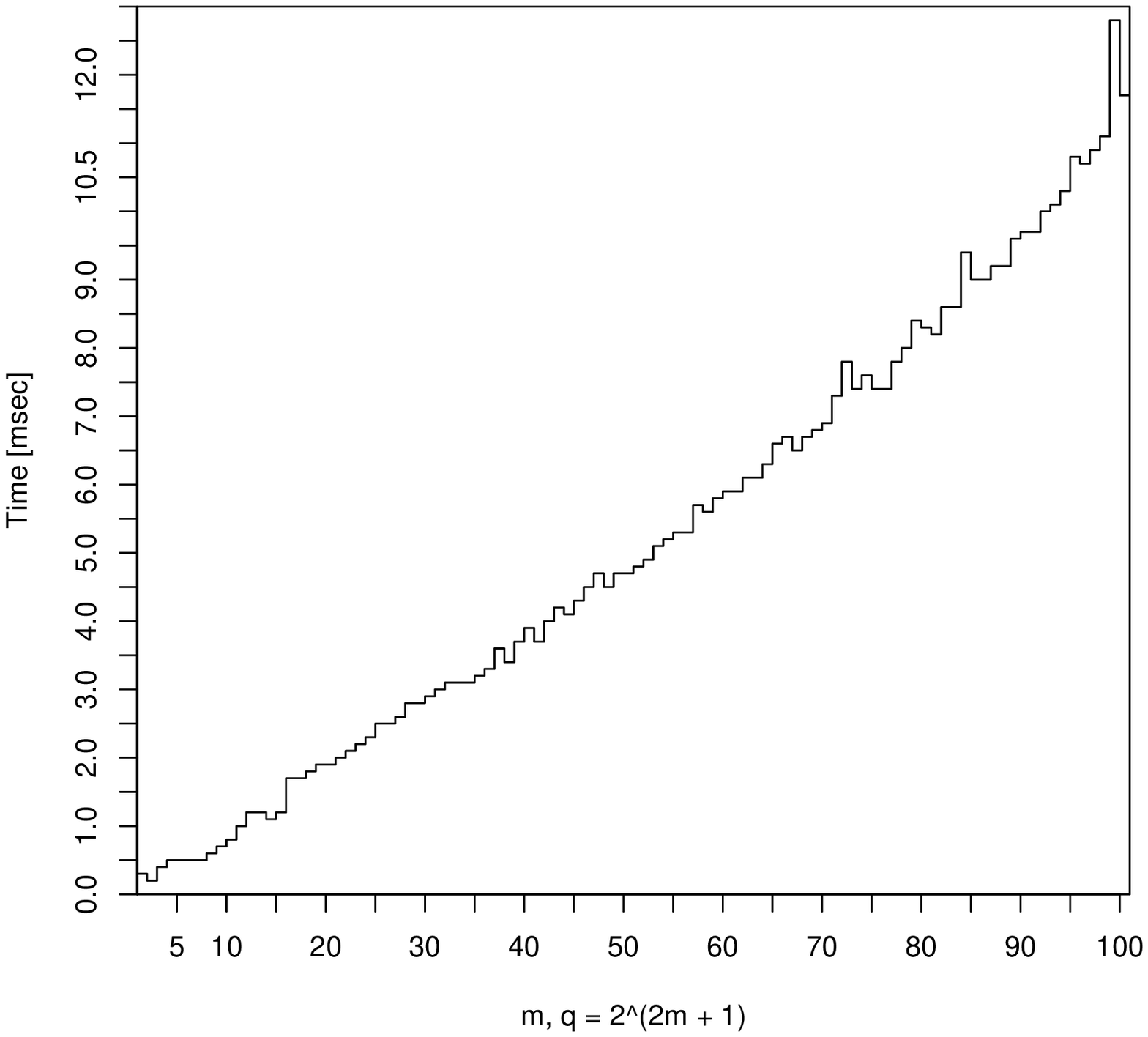}
\caption{Benchmark of membership testing}
\label{fig:membership_benchmark}
\end{figure}

In all benchmarks, we used a generating set of size $2$. We can always
switch to a generating set of this size by choosing random elements of
the input group. The probability that $2$ uniformly random elements
generate the group is high, and we can detect this using \cite[Theorem
$6.2$]{baarnhielm05}.

The benchmark was carried out using $\MAGMA$ V2.22-7, Intel64 CUDA 5.5
flavour, on a PC with an Intel Core i5-4690 CPU running at $3.5$
GHz. We used the software package \textsc{R} \cite{r_man} to produce
the figure.

\bibliographystyle{amsplain}
\bibliography{natsubm}

\end{document}